\documentclass[a4paper,12pt,oneside,reqno ]{amsart}

\usepackage[T1]{fontenc}
\usepackage[utf8]{inputenc}
\usepackage{lmodern}
\usepackage[english]{babel}

\usepackage[%
left=2.5cm,       
right=2.5cm,      
top=3.5cm,        
bottom=3.5cm,     
heightrounded,    
bindingoffset=0mm 
]{geometry}
\usepackage{amsrefs}
\usepackage{amsmath}
\usepackage{amssymb}
\usepackage{mathtools}
\usepackage{mathrsfs}
\usepackage[colorlinks=true,linkcolor=blue,urlcolor=blue,citecolor=blue]{hyperref}

\numberwithin{equation}{section}

\usepackage{hyperref} 

\usepackage{tikz}
\usepackage{graphicx}
\usepackage{xcolor}
\usepackage[font=small]{caption}

\usepackage{bookmark}

\theoremstyle{plain}
\newtheorem{theorem}{Theorem}[section]
\newtheorem{lemma}[theorem]{Lemma}
\newtheorem{proposition}[theorem]{Proposition}

\theoremstyle{definition}
\newtheorem{definition}[theorem]{Definition}

\newtheorem{remark}[theorem]{Remark}
\newtheorem{open.problem}[theorem]{Open Problem}

\newcommand{\N}{\mathbb{N}}
\newcommand{\R}{\mathbb{R}}

\title[Minimization of the buckling load with perimeter constraint]{Minimization of the buckling load of a clamped plate with perimeter constraint}

\author[M.\ Carriero]{Michele Carriero}
\address{Dipartimento di Matematica
	e Fisica ``E. De Giorgi'', Universit\`a del Salento,
	Via Per Arnesano, 73100 Lecce, Italy.}
\email{michele.carriero@unisalento.it}

\author[S.\ Cito]{Simone Cito}
\address{Dipartimento di Matematica
	e Fisica ``E. De Giorgi'', Universit\`a del Salento,
	Via Per Arnesano, 73100 Lecce, Italy.}
\email{simone.cito@unisalento.it}

\author[A.\ Leaci]{Antonio Leaci}
\address{Dipartimento di Matematica
	e Fisica ``E. De Giorgi'', Universit\`a del Salento,
	Via Per Arnesano, 73100 Lecce, Italy.}
\email{antonio.leaci@unisalento.it}

\date{\today}  



\begin{document}
\begin{abstract}
		We look for minimizers of the buckling load problem with perimeter constraint in any dimension. In dimension 2, we show that the minimizing plates are convex; in higher dimension, by passing through a weaker formulation of the problem, we show that any optimal set is open and connected. For higher eigenvalues, we prove that minimizers exist among convex sets with prescribed perimeter.
\end{abstract}
	
	\maketitle
	
	\tableofcontents
	
	\section{Introduction}
	\label{sec:intro}
	Let $d\in\N$, $d\ge 2$ and let $\Omega\subset\R^d$ a bounded Lipschitz domain. We say that $\Lambda(\Omega)$ is an eigenvalue of the buckling load problem (briefly, a buckling eigenvalue) if there exists $u\in H^2_0(\Omega)\setminus\left\{0\right\}$ such that
	\begin{equation}\label{eq:buckpde}
	\begin{cases}
	-\Delta^2 u=\Lambda(\Omega)\Delta u &\text{in $\Omega$},\\
	u=\frac{\partial u}{\partial \nu}=0 &\text{on $\partial\Omega$}.
	\end{cases}
	\end{equation}
The buckling eigenvalues form an increasing sequence
$$0<\Lambda_1(\Omega)\le\Lambda_2(\Omega)\le\ldots\nearrow+\infty$$
and, for any $h\in\N$, they can be characterized variationally by the min-max formula involving the Rayleigh quotient
	$$\Lambda_h(\Omega)=\min_{V\subset H_0^2(\Omega),\dim V=h}\max_{u\in V\setminus\left\{0\right\}}\frac{\displaystyle\int_\Omega(\Delta u)^2\:dx}{\displaystyle\int_\Omega |\nabla u|^2\:dx}.
$$
In this paper we mainly focus on the first eigenvalue, i.e.
	\begin{equation}\label{eq:buckvar}
	\Lambda_1(\Omega)=\min_{u\in H_0^2(\Omega)\setminus\left\{0\right\}}\frac{\displaystyle\int_\Omega(\Delta u)^2\:dx}{\displaystyle\int_\Omega |\nabla u|^2\:dx}.
	\end{equation}
The minimum above is achieved only on the solutions of Problem \eqref{eq:buckpde}.

Our aim is to show an existence result for a shape optimization problem involving $\Lambda_1(\Omega)$, whose formulation is appropriated in view of the physical interpretation of the PDE. Indeed, in a $2$-dimensional setting, $\Omega$ can be thought as a thin elastic plate that is clamped along its boundary $\partial\Omega$ and subject to compressive forces (the so called ``buckling forces'') across $\partial\Omega$; these forces lie on the same plane as $\Omega$, that may deflect out of its plane when these forces reach a
certain magnitude. $\Lambda_1(\Omega)$ is said the ``buckling load'' of $\Omega$ and can be interpreted as the energy associated to the plate $\Omega$ in this phenomenon.\\

\begin{figure}[h]
	\centering
		\includegraphics[width=0.7\textwidth]{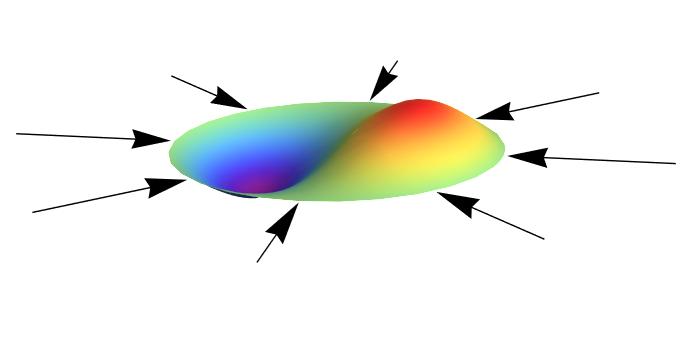}
		\caption{A clamped plate subject to the buckling forces in dimension $d=2$; the deflection can be both upward and downward.}
\end{figure}

There are some works in literature treating this problem by constraining the volume of the admissible sets and there is a few information about the minimizers. The first result that uses variational methods is \cite{AshBuc2003}, where the authors prove that the problem in $\R^2$ admits a quasi open minimizer without prescribing any bounded design region; moreover, supposing that a minimizer $\Omega$ is a sufficiently smooth domain with sufficiently smooth eigenfunctions, the authors show that $\Omega$ has to be a disk of maximal area. Several years later, in \cite{sto16} it is proved an existence result in dimension 2 and 3 with a different technique based on the eigenfunctions, but introducing a (big) bounded design region to assure extra compactness. 
Recently, \cite{sto2021} provided an existence result for minimizers in the whole of $\R^d$ based on a mixed strategy: a concentration-compactness argument inspired by \cite{AshBuc2003} to get a limit function $u$ and a regularity argument for $u$ to build the optimal open set using the superlevel sets of continuous functions; moreover, the author shows that the open minimizers are also connected, but nothing is proved about the regularity of the boundary. On the other hand, it seems that there are no available results in literature for higher eigenvalues, even if the admissible shapes satisfy some topological constraint.

Nevertheless, in view of the physical interpretation of the buckling eigenvalues, it seems reasonable to replace the volume constraint with the perimeter constraint; in this work, given $p>0$, we focus on the following problem
\begin{equation}\label{eq:minbuck}
\min\left\{\Lambda_1(\Omega):\Omega\subset\R^d\ \text{open},\ |\Omega|<\infty,\ P(\Omega)\le p\right\}.
\end{equation}
An appropriate interpretation in $\R^2$ can be the following: given a deformable support (with prescribed length $p$) wherein the admissible plates can be clamped, we want to find if there exists a plate that minimizes the buckling load due to the buckling forces acting across the support itself; in other words, we look for the optimal plate that can be clamped in the given support. The main result of the paper is the following.

	\begin{theorem}\label{teo:maintheorem}
	Problem \eqref{eq:minbuck} admits a solution. Every minimizer is an open connected set. In $\R^2$ any optimal set is open, bounded and convex.
	\end{theorem}
	
We also prove the following existence result for higher eigenvalues in the framework of convex sets.
\begin{theorem}\label{pro:kconv}
Problem \begin{equation}\label{eq:kconv}
\min\left\{\Lambda_h(\Omega):\Omega\subset\R^d\ \text{open and convex}, \mathcal{H}^{d-1}(\partial\Omega)\le p\right\},
\end{equation}
admits a solution.
\end{theorem}

We finally point out that spectral shape optimization problems governed by higher order PDEs are, in general, harder to handle than the second order counterparts and non existence results often appear. To get acquainted on the non existence of optimal shapes for higher order problems, see several results in \cite{GGS} (and the reference therein). We just mention, for instance, \cite[Theorem 3.24]{GGS}, which gives an example of non existence of optimal shapes for a spectral functional neither with volume nor with perimeter constraint among convex sets.
	
The paper is structured as follows. In Section \ref{sec:prel} we give some preliminary results and we fix the notation. In Section \ref{sec:d=2} we treat separately the existence case in dimension 2 (it is almost straightforward, but we insert the proof for the reader's convenience). In Section \ref{sec:d>2} we deal with the existence in any dimension in a weak framework and we show that the optimal shapes for the weak problem are in fact minimizers for the original problem, completing the proof of Theorem \ref{teo:maintheorem}. In Section \ref{sec:kconv} we prove Theorem \ref{pro:kconv}. We finally set some open problems and perspectives in Section \ref{sec:open}.

\section{Notation and Preliminary Results}\label{sec:prel}

In this section we fix the notation and recall some results used throughout the paper.

For $x\in\mathbb{R}^d$ and $r>0$, $B_r(x)$ will denote the open ball of radius $r$ centered in $x$; when $x$ is omitted, we consider the ball centered in the origin. For every measurable set $E\subseteq\mathbb{R}^d$,  we will use the symbols $\chi_E$ for the characteristic function of $E$, $E^c$ for its complement and $tE$ for the rescaled set $\left\{tx:x\in E\right\}$. As usual, $|E|$ and $\mathcal{H}^s(E)$ ($s>0$) stand respectively for the Lebesgue measure and the Hausdorff $s$-dimensional measure of $E$; if $E$ is a piecewise regular hypersurface, $\mathcal{H}^{d-1}(E)$ coincides with its area measure. We will denote by $\mathcal{H}-dim(E)$ the Hausdorff dimension of the set $E$; for sufficiently regular sets, it coincides with the topological dimension of the set $E$, e.g. if $E$ is an open set of $\mathbb{R}^d$, $\mathcal{H}-dim(E)=d$ (for further details see Chapter 2, Section 8 in \cite{AFP}). For every open set $\Omega\subset\R^d$, we will denote by $L^p(\Omega)$ the usual Lebesgue space of (classes of) $p$-summable functions, by $W^{k,p}(\Omega)$ the Sobolev space of functions whose (weak) derivatives are $p$-summable up to order $k$ and by $H^k(\Omega)$ the (Hilbert) space $W^{k,2}(\Omega)$; whenever $f\in L^p(K)$ or $f\in W^{k,p}(K)$ for any compact set $K\subset\Omega$, we say that $f\in L^p_{loc}(\Omega)$ or $f\in W_{loc}^{k,p}(\Omega)$ respectively. For the convenience of the reader, whenever $A,B$ are open sets, $A\subseteq B$ and $u\in W^{k,p}_0(A)$, we will denote still by $u$ its zero extension to the whole $B$.

Moreover, for any open set $\Omega$ and any test function $u\in H^2_0(\Omega)$, we denote the Rayleigh quotient in \eqref{eq:buckvar} by $R_\Omega(u)$.

\begin{definition}[sets of finite perimeter]
Let $E\subseteq\mathbb{R}^d$ be measurable and let $\Omega\subseteq\mathbb{R}^d$ be open. We define the perimeter of $E$ in $\Omega$ as
$$P(E,\Omega):=\sup\left\{\int_E\text{div}(\varphi)\:dx : \varphi\in C^1_c(\Omega;\mathbb{R}^d),\|\varphi\|_\infty\le 1\right\}$$
and we say that $E$ is of finite perimeter in $\Omega$ if $P(E,\Omega)<+\infty$. If $\Omega=\mathbb{R}^d$ we simply say that $E$ is of finite perimeter and denote its perimeter by $P(E)$.
\end{definition}

Let us recall that, if $E$ is sufficiently regular (e.g. if $E$ is either a bounded Lipschitz domain or a convex set), it holds $P(E,\Omega)=\mathcal{H}^{d-1}(\partial E\cap\Omega)$.

In order to minimize Problem \eqref{eq:minbuck} and its weak version, Problem \eqref{eq:relaxminbuck}, using the direct methods of the Calculus of Variation (or some concentration-compactness argument), we need lower semicontinuity of the buckling eigenvalues with respect to the some suitable topology on the class of sets of $\mathbb{R}^d$ where the problem is set. As we will see, two good choices for our purposes are the Hausdorff topology of open sets (in the $2$-dimensional setting, where the perimeter constraint and the monotonicity of the functional imply the convexity of the optimizers) and the $L^1$-topology (in higher dimension, where a weak formulation in the class of measurable sets is needed).

\begin{definition}[convergence in measure]
We say that a sequence of measurable sets $(\Omega_n)_n$ \emph{converges in measure} to a measurable set $\Omega$ if $|\Omega_n\Delta\Omega|\to 0$, namely if $\chi_{\Omega_n}\to\chi_\Omega$ in $L^1(\R^d)$.

We say that $(\Omega_n)_n$ \emph{locally converges in measure} to $\Omega$ if $\chi_{\Omega_n}\to\chi_\Omega$ in $L^1_{loc}(\R^d)$.
\end{definition}

This kind of convergence turns out to be suitable to our purposes to have compactness of minimizing sequences of sets of finite perimeter.

\begin{proposition}[compactness for the convergence in measure and lower semicontinuity of the perimeter, \cite{AFP},Theorem 3.39]\label{pro:3.39}
Let $A\subset\mathbb{R}^d$ be an open bounded set and let $(E_n)_n$ be a sequence of subsets of $A$ with finite perimeter such that
$$\sup_n P(E_n,A)<+\infty.$$
Then, there exists $E\subseteq A$ with finite perimeter in $A$ such that, up to subsequences,
$$\chi_{E_n}\to\chi_E\quad\text{in $L^1(\R^d)$}$$
and
$$P(E,A)\le\liminf_{n\to\infty} P(E_n,A).$$
\end{proposition}

In general, one of the main disadvantages of the convergence in measure is that no topological properties of converging sequences can be deduced for the limit set in this framework. To this aim, we introduce the Hausdorff convergences.

\begin{definition}[Hausdorff topology on closed sets]
Let $A,B\subseteq \mathbb{R}^d$ be closed. We define the \emph{Hausdorff distance between $A$ and $B$} by
$$d_H(A,B):=\max\left\{\sup_{x\in A}\text{dist}(x,B),\sup_{x\in B}\text{dist}(x,A)\right\}.$$
The topology induced by this distance is called \emph{Hausdorff topology (or simply $H$-topology) on closed sets}.
\end{definition}

The counterpart of the Hausdorff topology for open sets is defined below.

\begin{definition}[Hausdorff topology on open sets]
Let $A,B\subseteq\mathbb{R}^d$ be open. We define the \emph{Hausdorff-complementary distance between $A$ and $B$} by
$$d_{H^c}(A,B):=d_H(A^c,B^c).$$
The topology induced by this distance is called \emph{Hausdorff-complementary topology (or simply $H^c$-topology) on open sets}.
\end{definition}

This topology guarantees the compactness of sequences of open convex sets under suitable hypotheses. The following proposition contains some results proved in \cite{bub}, Section 2.4, and shows us that Hausdorff convergences preserve convexity and assure continuity for Lebesgue measure and perimeters of convex sets.

\begin{proposition}\label{prop:convexcontinuity}
The following results hold for convex sets:
\begin{itemize}
\item[(i)] If $A\subseteq B$, then $\mathcal{H}^{d-1}(\partial A)\le\mathcal{H}^{d-1}(\partial B)$;
\item[(ii)] If $A_n, A$ are closed (respectively open) and convex and $A_n\to A$ with respect to the $H$-topology (respectively $H^c$-topology), then $\chi_{A_n}\to\chi_{A}$ in $L^1$; moreover, if for every $n\in\N$ it holds $\mathcal{H}-dim(A)=\mathcal{H}-dim(A_n)$, then $\mathcal{H}^{d-1}(\partial A_n)\to\mathcal{H}^{d-1}(\partial A)$.
\item[(iii)] $|A|\le\rho\mathcal{H}^{d-1}(\partial A)$, where $\rho$ is the radius of the biggest ball contained in $A$.
\item[(iv)] If a sequence $(A_n)_n$ of closed convex sets $H$-converges to a closed set $A$, then $A$ is a closed convex set; if a sequence $(B_n)_n$ of open convex sets $H^c$-converges to an open set $B$, then $B$ is an open convex set.
\item[(v)] Let $D\subset\mathbb{R}^d$ a fixed compact set. Then, the class of the closed convex sets contained in $D$ is compact in the $H$-topology and the class of the open convex sets contained in $D$ is compact in the $H^c$-topology 
\end{itemize}
\end{proposition}

We recall an important result due to F. John (see \cite{joh}), involving convex sets.

\begin{theorem}[John's ellipsoid Theorem]\label{teo:john}
Let $K\subset\R^d$ a compact convex set with non-empty interior. Then, there exists an ellipsoid $E\subset\R^d$ centered in $x_0\in E$ such that
$$E\subseteq K\subseteq x_0+d(E-x_0)$$
(where the ellipsoid $x_0+d(E-x_0)$ is obtained by a dilation of $E$ by a factor $d$ and with center $x_0$).
\end{theorem}

The following properties of the buckling eigenvalues will be useful throughout the paper.

\begin{proposition}[properties of the buckling eigenvalues]
\textcolor{white}{1}
\begin{itemize}
\item[(i)] Let $\Omega_1,\Omega_2\subset\R^d$ be open sets such that $\Omega_1\subset\Omega_2$; then, for any $h\in\N$,
$$\Lambda_h(\Omega_2)\le\Lambda_h(\Omega_1).$$
\item[(ii)] For every set of finite perimeter $\Omega\subset\R^d$ and $t>0$
$$\Lambda_h(t\Omega)=t^{-2}\Lambda_h(\Omega).$$
\end{itemize}
\begin{proof}
The proof of item (i) is straightforward since $H^2_0(\Omega_1)\subseteq H^2_0(\Omega_2)$. Item (ii) can be proved via the natural change of variables $t\Omega\ni x\mapsto y=\frac{x}{t}\in \Omega$ in the Rayleigh quotient and the one-to-one correspondence between $u\in H^2_0(t\Omega)$ and $u(t\,\cdot)\in H^2_0(\Omega)$.
\end{proof}
\end{proposition}

\begin{remark}[equivalent formulations]\label{rem:equiv}
In view of the scaling properties of the perimeter and of the buckling eigenvalues, Problem \eqref{eq:minbuck} is equivalent to the scale invariant problem
\begin{equation}\label{eq:scaleinv}
\min\left\{P(\Omega)^\frac{2}{d-1}\Lambda_1(\Omega):\Omega\subset\R^d\ \text{open},\ |\Omega|<\infty,\right\}.
\end{equation}
Indeed, for any $t>0$ one has
$$P(t\Omega)^\frac{2}{d-1}\Lambda_1(t\Omega)=(t^{d-1})^\frac{2}{d-1} P(\Omega)^\frac{2}{d-1}{t^{-2}}\Lambda_1(\Omega)=P(\Omega)^\frac{2}{d-1}\Lambda_1(\Omega)$$
Moreover, Problem \eqref{eq:minbuck} is also equivalent to the penalized problem
\begin{equation}\label{eq:penalized}
\min\left\{\Lambda_1(\Omega)+\beta P(\Omega):\Omega\subset\R^d\ \text{open},\ |\Omega|<\infty,\right\}
\end{equation}
for some $\beta>0$. More precisely, if $\hat{\Omega}$ is a solution of Problem \eqref{eq:minbuck}, there exists $\beta>0$ such that $\hat{\Omega}$ is a solution of Problem \eqref{eq:penalized} and, viceversa, if $\hat{\Omega}$ is a solution of Problem \eqref{eq:penalized}, then it solves Problem \eqref{eq:minbuck} with bound on the perimeters given by $p=P(\hat{\Omega})$.

The second implication is straightforward. To prove the equivalence, then, it is sufficient to consider a solution $\hat{\Omega}$ of Problem \eqref{eq:minbuck}, define the function on $\R_+$
$$F(t):=\Lambda_1(t\hat{\Omega})+\beta P(t\hat{\Omega})$$
and show that it attains its minimum in $t=1$. By the scaling properties of the perimeter and of the eigenvalues we have
$$F(t)=t^{-2}\Lambda_1(\hat{\Omega})+\beta t^{d-1}P(\hat{\Omega}).$$
To conclude, we choose $\beta>0$ in such a way that the derivative
$$F'(t)=-2t^{-3}\Lambda_1(\hat{\Omega})+\beta(d-1)t^{d-2}P(\hat{\Omega})$$
vanishes in $t=1$, i.e.
$$\beta=\frac{2\Lambda_1(\hat{\Omega})}{(d-1)P(\hat{\Omega})}.$$
\end{remark}

We close this section recalling two useful results involving some properties of the Sobolev spaces.
\begin{proposition}[see \cite{evans2018measure}, Theorem 4.4.(iv)]\label{teo:evans4.4}
Let $1\le p<\infty$ and let $f\in W^{1,p}(\R^d)$. Then $\nabla f=0$ a.e. on $\{f=0\}$.
\end{proposition}

\begin{proposition}[see \cite{adams1999function}, Theorem 9.1.3]\label{teo:adams9.1.3}
Let $m$ be a positive integer, let $1<p<\infty$ and let $f\in W^{m,p}(\R^d)$. Let $\Omega\subset\R^d$ be an open set. Then the following statements are equivalent:
\begin{itemize}
\item[(a)] $D^\alpha f=0$ everywhere in $\Omega^c$ for all multiindices $\alpha$ such that $0\le|\alpha|\le m - 1$;
\item[(b)] $f\in W^{m,p}_0(\Omega)$.
\end{itemize}
\end{proposition}

\section{Existence of optimal shapes for the first buckling eigenvalue: the planar case}\label{sec:d=2}

We are able to prove the existence result in dimension two, where the perimeter constraint assures compactness.

\begin{proposition}\label{pro:exis2}
Problem \eqref{eq:minbuck} admits a bounded, open, convex minimizer $\Omega\subset\R^2$ with maximal boundary length.
\begin{proof}
We first notice some simplifications that can be done.
\begin{itemize}
\item Since $\Omega\mapsto\Lambda_1(\Omega)$ is invariant under translations of the connected components, we can suppose that they lie at zero distance each other; for the same reason, we can suppose that all the admissible shapes $\Omega$ are contained in the same bounded design region. Indeed, for any open set $\Omega\subset\R^2$ the condition $P(\Omega)\le p$ implies $\text{diam}(\Omega)<p/2$ and then all the admissible domains can be translated in a bounded design region $D\subset\subset\R^2$.
\item For any admissible $\Omega$, its convex hull $\tilde{\Omega}$ is still an admissible set, since it is open and $\mathcal{H}^{1}(\partial\tilde{\Omega})\le \mathcal{H}^{1}(\partial\Omega)$. In view of the decreasing monotonicity of the map $\Omega\mapsto\Lambda_1(\Omega)$ with respect to the set inclusion, since $\Omega\subset\tilde{\Omega}$, we have
$$\Lambda_1(\tilde{\Omega})\le\Lambda_1(\Omega).$$
Then, we can reduce ourselves to the class of open convex sets with boundary length less than or equal to $p$.
\item In view of the scaling property of $\Lambda_1$ and the monotonicity with respect to inclusions, we have that the admissible sets can be assumed to have exactly boundary length equal to $p$ (the perimeter constraint is saturated).
\end{itemize}
In other words, in $\R^2$, without loss of generality we can study
\begin{equation}\label{eq:minbuck2}
\min\left\{\Lambda_1(\Omega):\Omega\subset D,\Omega\ \text{open and convex}, \mathcal{H}^{1}(\partial\Omega)=p\right\}.
\end{equation}
From now on, we denote $\mathcal{A}_p:=\{\Omega\subset D,\Omega\ \text{open and convex}, \mathcal{H}^{1}(\partial\Omega)=p\}$. Let now $(\Omega_n)_n$ be a minimizing sequence for \eqref{eq:minbuck2} and, for any $n\in\N$, let $u_n\in H^2_0(\Omega_n)$ an eigenfunction for $\Lambda_1(\Omega_n)$ with $\|\nabla u_n\|_{2}=1$. In view of the properties of the Hausdorff convergence (Proposition \ref{prop:convexcontinuity}), there exist a subsequence $(\Omega_{n_k})_k$ and an open convex set $\Omega\subset D$ such that
$$\Omega_{n_k}\to\Omega\quad\text{in the sense of Hausdorff};$$
notice that the convergence is also in measure and that $\Omega\neq\emptyset$. Otherwise, in view of the convergence in measure, we would have $|\Omega_{n_k}|\to 0$ and from the Payne inequality (see, for instance, inequality (3.26) in \cite{GGS}), we would obtain
$$\Lambda_1(\Omega_{n_k})\ge\lambda_2(\Omega_{n_k})\rightarrow+\infty,$$
where $\lambda_2$ is the second eigenvalue of the Dirichlet-Laplacian. This would contradict the minimality of the sequence $(\Omega_n)_n$. Then, $\Omega$ is an open convex set of positive measure and, in addition, it holds
$$\partial\Omega_{n_k}\to\partial\Omega\ \text{in the sense of Hausdorff}$$
and
$$\mathcal{H}^{1}(\partial\Omega)=\lim_{k\to+\infty}\mathcal{H}^{1}(\partial\Omega_{n_k})=p,$$
so $\Omega$ is an admissible set for Problem \eqref{eq:minbuck2}. Now, the corresponding subsequence of eigenfunctions $(u_{n_k})_k$ is bounded in $H^2_0(D)$; indeed
$$\|u_{n_k}\|_{H^2_0(D)}=\|u_{n_k}\|_{H^2_0(\Omega_{n_k})}=\int_{\Omega_{n_k}}(\Delta u_{n_k})^2\:dx=\Lambda_1(\Omega_{n_k})\le C.$$
Then, there exist a further subsequence (still denoted with the same index) and a function $u\in H^2_0(D)$ such that $u_{n_k}\rightharpoonup u$ weakly in $H^2_0(D)$; this implies that $u_{n_k}\to u$ strongly in $H^1_0(D)$. Moreover, since $\Omega_{n_k}$ converges to $\Omega$ in measure, we deduce that $u\in H^2_0(\Omega)$ is an admissible test function for $\Lambda_1(\Omega)$.

In view of the lower semicontinuity of the $H^2_0$-norm with respect to the weak convergence in $H^2_0(D)$ it finally holds
$$\Lambda_1(\Omega)\le\int_{\Omega}(\Delta u)^2\:dx\le\liminf_{k\to+\infty}\int_{\Omega_{n_k}}(\Delta u_{n_k})^2\:dx=\liminf_{k\to+\infty}\Lambda_1(\Omega_{n_k})=\inf_{\Omega\in\mathcal{A}_p}\Lambda_1(\Omega),$$
proving the thesis.
\end{proof}
\end{proposition}


\section{Existence of optimal shapes for the first buckling eigenvalue: the case of higher dimension}\label{sec:d>2}
 
In general, the perimeter constraint does not imply the boundedness and the convexity of optimal shapes in higher dimension (this is a peculiarity of the 2-dimensional case). To prove the existence of optimal shapes in higher dimension, we follow a different strategy. By following the approach of \cite{AshBuc2003} (later used in \cite{sto2021}), we look for minimizers for Problem \eqref{eq:minbuck} with neither topological constraint nor bounded design region via a concentration-compactness argument. The main difference is that in the previous works the authors dealt with a measure constraint, whereas we have to preserve a perimeter constraint. Proposition \ref{pro:3.39} suggests a good tool to this aim, since the result guarantees the lower semicontinuity of the perimeter and the compactness with respect to the convergence in measure for a sequence of measurable sets; in view of this, the original framework of open sets does not seem to be the most appropriate to prove an existence result for Problem \eqref{eq:minbuck}.

A good strategy in this sense is provided in \cite{deve}, where the authors use a suitable weak formulation of the Dirichlet-Laplacian eigenvalues in the framework of sets of finite perimeter. More precisely, in order to set the problem in the class of measurable sets instead of working only with open sets, for every set of finite perimeter $\Omega\subset\R^d$ they introduce the Sobolev-like spaces
$$\tilde{H}_0^1(\Omega):=\left\{u\in H^1(\R^d):u=0\ \text{a.e. in}\ \Omega^c\right\}$$
and they prove the existence of optimal shapes for a weaker version of the functional with perimeter constraint. After proving the existence, they are able to come back to the original problem, showing that weak minimizers are, in fact, open sets.

\begin{definition}[weak eigenvalues]
Let $\Omega\subset\R^d$ be a set of finite perimeter. We define the Sobolev-like space $\tilde{H}_0^2(\Omega)\subset H^2(\R^d)$ as
\begin{equation}\label{eq:htilde20}
\tilde{H}_0^2(\Omega):=\left\{u\in H^2(\R^d):u=0\ \text{a.e. in}\ \Omega^c\right\}.
\end{equation}
We define the $h$-th weak buckling eigenvalue by
$$\tilde{\Lambda}_h(\Omega)=\inf_{V\subset \tilde{H}_0^2(\Omega),\dim V=h}\max_{u\in V\setminus\left\{0\right\}}\frac{\displaystyle\int_\Omega(\Delta u)^2\:dx}{\displaystyle\int_\Omega |\nabla u|^2\:dx}.$$
In particular, the first weak buckling eigenvalue of $\Omega$ is given by
\begin{equation}\label{eq:relaxbuckvar}
	\tilde{\Lambda}_1(\Omega)=\inf_{u\in \tilde{H}_0^2(\Omega)\setminus\left\{0\right\}}\frac{\displaystyle\int_\Omega(\Delta u)^2\:dx}{\displaystyle\int_\Omega |\nabla u|^2\:dx},
\end{equation}
\end{definition}

Once given the weaker version of the functional, we look for a right class of admissible sets; our choice is the following:
\begin{equation*}
\tilde{\mathcal{A}}_p:=\left\{\Omega\subset\R^d\ \text{measurable},\ |\Omega|<\infty,\ P(\Omega)\le p\right\}.
\end{equation*}
The new problem to consider is thus
\begin{equation}\label{eq:relaxminbuck}
\min\left\{\tilde{\Lambda}_1(\Omega):\Omega\in\tilde{\mathcal{A}}_p\right\}.
\end{equation}
The choice of this weaker framework has been made in order to ensure the completeness of the class of admissible sets with respect to the local convergence in measure: in other words, a converging sequence of admissible sets converges (locally in measure) to an admissible set. Notice that, in the original statement \eqref{eq:minbuck}, this request fails: the limit set of a sequence of open sets converging in measure is not open, in general. For that reason, it has been necessary to choose also the functional in a weaker sense, keeping into account the new functional space for the test functions. Only the inclusion $\tilde{H}_0^2(\Omega)\supseteq H_0^2(\Omega)$ is valid in general, even for open sets. Nevertheless, if the set $\Omega$ is open and sufficiently regular, the equality $\tilde{H}_0^2(\Omega)=H_0^2(\Omega)$ holds; the equality fails whenever $\Omega$ has inner cracks, e.g. if $\Omega\subset\R^3$ is a ball with an equatorial cut that removes a maximal half-disk, namely
$$\Omega=B_1(0)\setminus\{x_3=0, x_1\ge 0\}.$$
Then, in general, for any open set $\Omega$ it holds
$$\tilde{\Lambda}_h(\Omega)\le\Lambda_h(\Omega).$$

Notice that also in Problem \eqref{eq:relaxminbuck} we avoid the apriori prescription of a bounded design region where the admissible sets are contained. A similar assumption would lead straightforwardly to the compactness of a minimizing sequence of admissible sets, see \cite[Theorem 3.39]{AFP}.

We now state some useful properties of the weak eigenvalues.

\begin{proposition}[properties of the weak eigenvalues]\textcolor{white}{1}
\begin{itemize}
\item[(i)] Let $\Omega_1,\Omega_2\subset\R^d$ be sets of finite perimeter such that $|\Omega_1\Delta\Omega_2|=0$; then, for any $k\in\N$,
$$\tilde{\Lambda}_h(\Omega_1)=\tilde{\Lambda}_h(\Omega_2).$$
\item[(ii)] Let $\Omega_1,\Omega_2\subset\R^d$ be sets of finite perimeter such that $|\Omega_2\setminus\Omega_1|=0$ (i.e. $\Omega_1\subset\Omega_2$ up to a $\mathcal{L}^d$-negligible set); then, for any $h\in\N$,
$$\tilde{\Lambda}_h(\Omega_2)\le\tilde{\Lambda}_h(\Omega_1).$$
\item[(iii)] For every set of finite perimeter $\Omega\subset\R^d$ and $t>0$
$$\tilde{\Lambda}_h(t\Omega)=t^{-2}\tilde{\Lambda}_h(\Omega).$$
\end{itemize}
\begin{proof}
Item (i) follows by observing that $|\Omega_1\Delta\Omega_2|=0$ implies $\tilde{H}^2_0(\Omega_1)=\tilde{H}^2_0(\Omega_2)$. Items (ii) and (iii) are proven in the same way as in the classical framework.
\end{proof}
\end{proposition}

\begin{remark}[equivalent formulations]\label{rem:relaxequiv}
As we proved for Problem \eqref{eq:minbuck}, also Problem \eqref{eq:relaxminbuck} has two equivalent formulations; more precisely it is equivalent to the scale invariant problem 
\begin{equation}\label{eq:relaxscaleinv}
\min\left\{P(\Omega)^\frac{2}{d-1}\tilde{\Lambda}_1(\Omega):\Omega\subset\R^d\ \text{measurable},\ |\Omega|<\infty,\right\}.
\end{equation}
and to the penalized problem
\begin{equation}\label{eq:relaxpenalized}
\min\left\{\tilde{\Lambda}_1(\Omega)+\beta P(\Omega):\Omega\subset\R^d\ \text{measurable},\ |\Omega|<\infty,\right\}
\end{equation}
for some $\beta>0$. The proof is the same as in Remark \ref{rem:equiv}.
\end{remark}

\begin{remark}[$\varepsilon$-eigenfunctions]
We point out that the infimum in \eqref{eq:relaxbuckvar} is not attained, in general. For that reason, we introduce the term `$\varepsilon$-eigenfunction' to denote a test function $u^\varepsilon\in \tilde{H}_0^2(\Omega)$ that satisfies
\begin{equation}\label{eq:uepsilon}
\tilde{\Lambda}_1(\Omega)\le\frac{\displaystyle\int_\Omega(\Delta u^\varepsilon)^2\:dx}{\displaystyle\int_\Omega|\nabla u^\varepsilon|^2\:dx}<\tilde{\Lambda}_1(\Omega)+\varepsilon.
\end{equation}
for some $\varepsilon>0$.
\end{remark}

The following result is a version of \cite[Lemma 3.3]{bucvar} adapted to our Sobolev-like spaces.

\begin{lemma}\label{pro:bucvar}
Let $(w_n)_n$ be a bounded sequence in $H^1(\R^d)$ such that $\|w_n\|_{L^2(\R^d)}=1$ and $w_n=0$ a.e. in $\Omega_n^c$ with $|\Omega_n|\le C$. There exists a sequence of vectors $(y_n)_n\subset\R^d$ such that the sequence $(w_n(\cdot+y_n))_n$ does not possess a weakly convergent subsequence in $H^1(\R^d)$.
\end{lemma}

We use the previous result to get the contradiction in the vanishing case in Theorem \ref{teo:relax} as follows: we find a sequence $(w_n)_n$ in $H^1(\R^d)$ satisfying (after a possible rescaling) $\|w_n\|_{L^2(\R^d)}=1$, $w_n=0$ a.e. in $\Omega_n^c$ with $|\Omega_n|\le C$ and such that any possible translation of $w_n$ weakly converges; in view of Lemma \ref{pro:bucvar} $(w_n)_n$ can not be bounded in $H^1(\R^d)$ and in particular gradients must be unbounded in $L^2(\R^d)$.

Now, to prove the existence of minimizers for Problem \eqref{eq:relaxminbuck}, we follow a strategy based on the concentration-compactness principle by P.-L. Lions (see \cite{lions84}) and inspired by \cite{AshBuc2003} (a similar argument has been reprised in \cite{sto2021}). We only have to be careful with the choice of suitable $\varepsilon$-eigenfunctions and to preserve the perimeter constraint. 
For another result of spectral shape optimization under perimeter constraint by using a concentration-compactness argument see \cite{bbh}, where the same technique applies in the minimization of the second Dirichlet-Laplace eigenvalue.

The main result of this section is the following.

\begin{theorem}\label{teo:relax}
Problem \eqref{eq:relaxminbuck} admits a measurable solution $\hat{\Omega}\subset\R^d$ with $P(\hat{\Omega})=p$.
\begin{proof}
First of all, we show that every solution $\hat{\Omega}\subset	\R^d$ has maximal perimeter. Otherwise, if $P(\hat{\Omega})<p$, the dilated set 
$$\hat{\Omega}':=\left(\frac{p}{P(\hat{\Omega})}\right)^{\frac{1}{d-1}}\hat{\Omega}$$
satisfies $P(\hat{\Omega}')=p$ and, in view of the decreasing monotonicity and the scaling property of $\tilde{\Lambda}_1$, it holds $\tilde{\Lambda}_1(\hat{\Omega}')<\tilde{\Lambda}_1(\hat{\Omega})$, leading to a contradiction with the minimality of $\hat{\Omega}$.

Now we prove the existence of a minimizer for Problem \eqref{eq:relaxminbuck}. Let $(\Omega_n)_n\subset\tilde{\mathcal{A}}_p$ be a minimizing sequence for Problem \eqref{eq:relaxminbuck} and let $u_n\in\tilde{H}_0^2(\Omega_n)$ a corresponding sequence of normalized weak $1/n$-eigenfunctions, namely we have
\begin{gather}
\int_{\R^d}|\nabla u_n|^2\:dx=1,\notag\\
 \tilde{\Lambda}_1(\Omega_n)\le\int_{\R^d}(\Delta u_n)^2\:dx<\tilde{\Lambda}_1(\Omega_n)+\frac{1}{n},\label{eq:minimizingsequence}\\
\inf_{\Omega\in\tilde{\mathcal{A}}_p}\tilde{\Lambda}_1(\Omega)=\lim_{n\to+\infty}\tilde{\Lambda}_1(\Omega_n)=\lim_{n\to+\infty}\int_{\R^d}(\Delta u_n)^2\:dx.\notag
\end{gather}
We infer
\begin{align*}
\|u_n\|_{L^2(\R^d)}&=\|u_n\|_{L^2(\Omega_n)}\le|\Omega_n|^{\frac{2^*-2}{2^*\cdot 2}}\|u_n\|_{L^{2^*}(\Omega_n)}\le C\|u_n\|_{L^{2^*}(\R^d)}\le C'\|\nabla u_n\|_{L^2(\R^d)}\\
&=C'\|\nabla u_n\|_{L^2(\Omega_n)}=C',
\end{align*}
where we used first the inclusion $L^{2^*}(\Omega_n)\subset L^{2}(\Omega_n)$ and then the Sobolev-Gagliardo-Nirenberg inequality since $d>2$; all the constants are independent of $n$ since they depend only on the dimension $d$ and on $|\Omega_n|$, that is uniformly bounded. We remark that we do not use directly the Poincaré inequality to show the uniform bound on $\|u_n\|_{L^2(\Omega_n)}$ since $\Omega_n$ can possibly be unbounded in any direction and, even if $\Omega_n$ is bounded, the function $u_n$ could be outside $H^1_0(\Omega_n)$. Moreover, an integration by parts yields
\begin{align*}
1=\int_{\R^d}|\nabla u_n|^2\:dx=-\int_{\R^d} u_n\Delta u_n\:dx\le C'\left(\int_{\R^d}(\Delta u_n)^2\:dx\right)^{1/2},
\end{align*}
so $\int_{\R^d}(\Delta u_n)^2\:dx$ is also bounded from below away from zero; this implies that the infimum of Problem \eqref{eq:relaxminbuck} is strictly positive. Analogously, also $\|u_n\|_{L^2(\R^d)}$ is larger than a positive constant and this avoids the degeneracy of the eigenfunctions $u_n$ in $L^2(\R^d)$.

As already highlighted, due to the lack of a bounded design region $D\subset\subset\R^d$, we apply the concentration-compactness Lemma \cite[Lemma I.1]{lions84} to the sequence $(|\nabla u_n|^2)_n$:\\\\
There exists a subsequence $(u_{n_k})_k\subset H^2(\R^d)$ such that one of the three following situations occurs.
\begin{itemize}
\item[(i)] \emph{Compactness.} There exists a sequence of points $(y_k)_k\subset\R^d$ such that
\begin{equation}\label{eq:com}
\forall\varepsilon>0\quad\exists R>0\quad\text{s.t.}\quad\int_{B_R(y_k)}|\nabla u_{n_k}|^2\:dx\ge 1-\varepsilon.
\end{equation}
\item[(ii)] \emph{Vanishing.} For every $R>0$
\begin{equation}\label{eq:van}
\lim_{k\to+\infty}\left(\sup_{y\in\R^d}\int_{B_R(y)}|\nabla u_{n_k}|^2\:dx\right)=0.
\end{equation}
\item[(iii)] \emph{Dichotomy.} There exists $\alpha\in]0,1[$ such that, for every $\varepsilon>0$, there exist two bounded sequences $(u^1_k)_k,(u^2_k)_k\subset H^2(\R^d)$ such that
\begin{gather}
\lim_{k\to+\infty}\text{dist}(\text{supp}(u_k^1),\text{supp}(u_k^2))=+\infty,\label{eq:dic1}\\
\lim_{k\to+\infty}\int_{\R^d}|\nabla u^1_{k}|^2\:dx=\alpha,\quad \lim_{k\to+\infty}\int_{\R^d}|\nabla u^2_{k}|^2\:dx=1-\alpha,\label{eq:dic2}\\
\lim_{k\to+\infty}\int_{\R^d}\left[|\nabla u_{n_k}|^2-(|\nabla u^1_{k}|^2+|\nabla u^2_{k}|^2)\right]\:dx\le\varepsilon,\label{eq:dic3}\\
\liminf_{k\to+\infty}\int_{\R^d}\left[(\Delta u_{n_k})^2-((\Delta u^1_{k})^2+(\Delta u^2_{k})^2)\right]\:dx\ge 0.\label{eq:dic4}
\end{gather}
\end{itemize}
Let us show that only compactness does occur.
\begin{itemize}
\item \emph{Vanishing does not occur.} Let us assume that vanishing occurs. Then every possible translation of any partial derivative $\frac{\partial u_{n_k}}{\partial x_j}$ weakly converges to 0 in $L^2(\R^d)$. Indeed, in view of \eqref{eq:van}, for any $\phi\in C^\infty_c(\R^d)$, with $\text{supp}(\phi)$ contained in some closed ball $\overline{B_R(y)}$, we have
$$\left|\int_{\R^d}\frac{\partial u_{n_k}}{\partial x_j}\phi\:dx\right|\le\left(\int_{B_R(y)}|\nabla u_{n_k}|^2\:dx\right)^{1/2}\left(\int_{B_R(y)}\phi^2\:dx\right)^{1/2}\to 0.$$
Now, since $\int_{\R^d}|\nabla u_{n_k}|^2\:dx=1$, there exists (up to permutations) $l\in\left\{1,\ldots,d\right\}$ such that
$$\int_{\R^d}\left(\frac{\partial u_{n_k}}{\partial x_l}\right)^2dx\ge\frac{1}{d}.$$
Since $u_{n_k}$ belongs to $H^2(\R^d)$, by two integration by parts we get, for any $i,h\in\left\{1,\ldots,d\right\}$,
$$\int_{\R^d}\frac{\partial^2 u_{n_k}}{\partial^2 x_i}\frac{\partial^2 u_{n_k}}{\partial^2 x_h}\:dx=\int_{\R^d}\left(\frac{\partial^2 u_{n_k}}{\partial x_i\partial x_h}\right)^2dx.$$
Then
\begin{align*}
\int_{\R^d}(\Delta u_{n_k})^2\:dx&=\sum_{i,h=1}^d\int_{\R^d}\left(\frac{\partial^2 u_{n_k}}{\partial x_i\partial x_h}\right)^2dx\\
&\ge\sum_{i=1}^d\int_{\R^d}\left(\frac{\partial^2 u_{n_k}}{\partial x_i\partial x_l}\right)^2dx=\int_{\R^d}\left|\nabla\left(\frac{\partial u_{n_k}}{\partial x_l}\right)\right|^2dx.
\end{align*}
We thus obtain, by recalling \eqref{eq:minimizingsequence}
\begin{equation}\label{eq:absvanish}
\tilde{\Lambda}_1(\Omega_{n_k})>\int_{\R^d}(\Delta u_{n_k})^2\:dx-\frac{1}{n_k}\ge\frac{1}{d}\frac{\displaystyle\int_{\R^d}\left|\nabla\left(\frac{\partial u_{n_k}}{\partial x_l}\right)\right|^2dx}{\displaystyle\int_{\R^d}\left(\frac{\partial u_{n_k}}{\partial x_l}\right)^2dx}-\frac{1}{n_k}.
\end{equation}
We now apply Lemma \ref{pro:bucvar} to the sequence $\left(\frac{\partial u_{n_k}}{\partial x_l}\right)_k$. Any translation of $\frac{\partial u_{n_k}}{\partial x_l}$ weakly converges to 0 in $L^2(\R^d)$ and in $H^1(\R^d)$ as well (every translation of $u_{n_k}$ is bounded in $H^2(\R^d)$, so every translation $\frac{\partial u_{n_k}}{\partial x_l}$ is weakly convergent in $H^1(\R^d)$). In view of Lemma \ref{pro:bucvar}, $\left(\frac{\partial u_{n_k}}{\partial x_l}\right)_k$ can not be bounded in $H^1(\R^d)$ and we get
$$\int_{\R^d}\left|\nabla\left(\frac{\partial u_{n_k}}{\partial x_l}\right)\right|^2dx\to+\infty,$$
that is in contradiction with \eqref{eq:absvanish}, since $(\tilde{\Lambda}_1(\Omega_{n_k}))_k$ is a minimizing sequence for Problem \eqref{eq:relaxminbuck}.
\item \emph{Dichotomy does not occur.} Let us suppose that dichotomy occurs. Let $\alpha\in]0,1[$ as in the statement of the dichotomy case (iii) and let $\varepsilon>0$. The sequences $(u_k^1)_k,(u_k^2)_k\in H^2(\R^d)$ can be chosen as follows, see \cite{lions84} I.1 and \cite{AshBuc2003}. Let $\phi\in C^\infty_c(\R^d;[0,1])$ such that $\phi\equiv 1$ in $B_1(0)$ and $\phi\equiv0$ in $\R^d\setminus B_2(0)$. Let $(r_k)_k,(\rho_k)_k\subset\R_+$ two diverging sequences and define for any $x\in\R^d$
$$u_k^1(x):=\phi\left(\frac{x}{r_k}\right)u_{n_k}(x),\quad u_k^2(x):=\left(1-\phi\left(\frac{x}{\rho_k r_k}\right)\right)u_{n_k}(x).$$
Notice that
$$\text{supp}(u_k^1)\subseteq\overline{\Omega_{n_k}\cap B_{2r_k}(0)},\quad \text{supp}(u_k^2)\subseteq\overline{\Omega_{n_k}}\setminus B_{\rho_k r_k}(0)$$
(so, that choice satisfies \eqref{eq:dic1}).

In view of the previous choice, by using \eqref{eq:dic2}, \eqref{eq:dic3}, \eqref{eq:dic4} and the inequality
\begin{equation}\label{eq:rapporto}
\frac{a_1+a_2}{b_1+b_2}\ge\min\left\{\frac{a_1}{b_1},\frac{a_2}{b_2}\right\}\quad\forall a_1,a_2,b_1,b_2>0,
\end{equation}
we have (possibly switching $u_k^1$ and $u_k^2$)
\begin{equation}
\label{eq:dicabs1}
\begin{split}
\inf_{\Omega\in\tilde{\mathcal{A}}_p}\tilde{\Lambda}_1(\Omega)&=\lim_{k\to+\infty}\frac{\displaystyle\int_{\R^d}(\Delta u_{n_k})^2\:dx}{\displaystyle\int_{\R^d}|\nabla u_{n_k}|^2\:dx}\ge\limsup_{k\to+\infty}\frac{\displaystyle\int_{\R^d}\left[(\Delta u_k^1)^2+(\Delta u_k^2)^2\right]\:dx}{\displaystyle\varepsilon+\int_{\R^d}\left(|\nabla u_k^1|^2+|\nabla u_k^2|^2\right)\:dx}\\
&\ge\limsup_{k\to+\infty}\frac{\displaystyle\int_{\R^d}(\Delta u_k^1)^2\:dx}{\displaystyle\varepsilon+\int_{\R^d}|\nabla u_k^1|^2\:dx}=\limsup_{k\to+\infty}\frac{\displaystyle\int_{\R^d}(\Delta u_k^1)^2\:dx}{\displaystyle\int_{\R^d}|\nabla u_k^1|^2\:dx}\cdot\frac{\displaystyle\int_{\R^d}|\nabla u_k^1|^2\:dx}{\displaystyle\varepsilon+\int_{\R^d}|\nabla u_k^1|^2\:dx}\\
&=\frac{\alpha}{\varepsilon+\alpha}\limsup_{k\to+\infty}\frac{\displaystyle\int_{\R^d}(\Delta u_k^1)^2\:dx}{\displaystyle\int_{\R^d}|\nabla u_k^1|^2\:dx}\ge\frac{\alpha}{\varepsilon+\alpha}\limsup_{k\to+\infty}\tilde{\Lambda}_1(\Omega_{n_k}\cap B_{2r_k}(0)).
\end{split}
\end{equation}
Now, let us define
$$t_k:=\left(\frac{P(\Omega_{n_k})}{P(\Omega_{n_k}\cap B_{2r_k}(0))}\right)^\frac{1}{d-1}.$$
Clearly, $t_k\ge 1$; by using the scaling property of the perimeter we have
\begin{align*}
P\left(t_k(\Omega_{n_k}\cap B_{2r_k}(0))\right)&=t_k^{d-1}P(\Omega_{n_k}\cap B_{2r_k}(0))\\
&=\frac{P(\Omega_{n_k})}{P(\Omega_{n_k}\cap B_{2r_k}(0))}P(\Omega_{n_k}\cap B_{2r_k}(0))=P(\Omega_{n_k})\le p,
\end{align*}
i.e. the dilated set $t_k(\Omega_{n_k}\cap B_{2r_k}(0))$ is admissible for Problem \eqref{eq:relaxminbuck}.

Moreover, in view of the scaling property of the weak eigenvalues, it holds
$$\tilde{\Lambda}_1(\Omega_{n_k}\cap B_{2r_k}(0))=t_k^2\tilde{\Lambda}_1(t_k(\Omega_{n_k}\cap B_{2r_k}(0))).$$
By using this equality in \eqref{eq:dicabs1} we get
\begin{equation}\label{eq:dicabs2}
\begin{split}
\inf_{\Omega\in\tilde{\mathcal{A}}_p}\tilde{\Lambda}_1(\Omega)&\ge\frac{\alpha}{\varepsilon+\alpha}\limsup_{k\to+\infty}t^2_k\tilde{\Lambda}_1(t_k(\Omega_{n_k}\cap B_{2r_k}(0)))\\
&\ge\frac{\alpha}{\varepsilon+\alpha}\inf_{\Omega\in\tilde{\mathcal{A}}_p}\tilde{\Lambda}_1(\Omega)\limsup_{k\to+\infty}t^2_k.
\end{split}
\end{equation}
We claim that
\begin{equation}\label{eq:delta}
\limsup_{k\to+\infty}t^2_k=\delta> 1.
\end{equation}
We argue by contradiction. Let us suppose that $\displaystyle\limsup_{k\to+\infty}t^2_k=1$ and let us denote by $(t_{k_j})_j$ a subsequence of $(t_k)_k$ such that
$$\lim_{j\to+\infty}t^2_{k_j}=\limsup_{k\to+\infty}t^2_k=1.$$
We thus have
$$\lim_{j\to+\infty}P(\Omega_{n_{k_j}})=\lim_{j\to+\infty}P(\Omega_{n_{k_j}}\cap B_{2r_{k_j}}(0)).$$
This implies, in view of \eqref{eq:dic1}, that
$$\lim_{j\to+\infty}P(\Omega_{n_{k_j}}\setminus B_{\rho_kr_{k_j}}(0))=0$$
and so $|\Omega_{n_{k_j}}\setminus B_{\rho_kr_{k_j}}(0)|\to 0$. But this is a contradiction, since it would imply $\tilde{\Lambda}_1(\Omega_{n_{k_j}}\setminus B_{\rho_kr_{k_j}}(0))\to+\infty$, that is impossible in view of the estimate
$$\limsup_{k\to+\infty}\tilde{\Lambda}_1(\Omega_{n_k}\setminus B_{\rho_kr_k}(0))\le\limsup_{k\to+\infty}\frac{\displaystyle\int_{\R^d}(\Delta u_k^2)^2\:dx}{\displaystyle\int_{\R^d}|\nabla u_k^2|^2\:dx}\le\frac{C}{1-\alpha},$$
where we applied the second limit in \eqref{eq:dic2} and the fact that $\displaystyle\int_{\R^d}(\Delta u_k^2)^2\:dx$ is uniformly bounded by a positive constant in view of \eqref{eq:dic4}.

We conclude that \eqref{eq:delta} is true.

By \eqref{eq:dicabs2} we obtain
\begin{equation*}
\inf_{\Omega\in\tilde{\mathcal{A}}_p}\tilde{\Lambda}_1(\Omega)\ge\frac{\alpha}{\varepsilon+\alpha}\delta\inf_{\Omega\in\tilde{\mathcal{A}}_p}\tilde{\Lambda}_1(\Omega)
\end{equation*}
and then, in view of the arbitrariness of $\varepsilon>0$, we infer 
$$\inf_{\Omega\in\tilde{\mathcal{A}}_p}\tilde{\Lambda}_1(\Omega)\ge\delta\inf_{\Omega\in\tilde{\mathcal{A}}_p}\tilde{\Lambda}_1(\Omega)>\inf_{\Omega\in\tilde{\mathcal{A}}_p}\tilde{\Lambda}_1(\Omega),$$
a contradiction.
\end{itemize}
Since neither vanishing nor dichotomy occurs, we conclude that compactness takes place.

Then, there exists a sequence $(y_k)_k\subset\R^d$ such that, up to subsequences,
$$\|u_{n_k}(\cdot+y_k)\|_{H^2(\R^d)}\le C$$
and there exists $u\in H^2(\R^d)$
$$u_{n_k}(\cdot+y_k)\rightharpoonup u\ \text{weakly in $H^2(\R^d)$},\quad \|\nabla u\|_{L^2(\R^d)}=1.$$
The equality $\|\nabla u\|_{L^2(\R^d)}=1$ comes from \eqref{eq:com}, the arbitrariness of $\varepsilon>0$ and the weak lower semicontinuity of the $L^2$-norm of the gradient:
$$1-\varepsilon\le\|\nabla u\|_{L^2(\R^d)}\le\liminf_{k\to+\infty}\|\nabla u_{n_k}(\cdot+y_k)\|_{L^2(\R^d)}=\liminf_{k\to+\infty}\|\nabla u_{n_k}\|_{L^2(\R^d)}=1,$$
Then, since $u_{n_k}(\cdot+y_k)\to u$ strongly in $L^2(\R^d)$ and
$$\lim_{k\to+\infty}\|\nabla u_{n_k}(\cdot+y_k)\|_{L^2(\R^d)}=1=\|\nabla u\|_{L^2(\R^d)},$$
we deduce that $u_{n_k}(\cdot+y_k)\to u$ strongly in $H^1(\R^d)$.

We now prove that there exists a measurable set $\hat{\Omega}\subset\R^d$ such that
$$|\hat{\Omega}|<\infty,\quad P(\hat{\Omega})\le p,\quad u=0\ \text{a.e. in $\hat{\Omega}^c$},$$
in order to use $u\in\tilde{H}^2_0(\hat{\Omega})$ as a test function for $\tilde{\Lambda}_1(\hat{\Omega})$.

For any $j\in\N$, let us define the set $\hat{\Omega}^{(j)}$ as the limit in measure of the sequence $(\Omega_{n_k}-y_k)_k$ in $B_j(0)$. Notice that
$$P(\hat{\Omega}^{(j)};B_j(0))\le\liminf_{k\to+\infty}P(\Omega_{n_k}-y_k;B_j(0))\le p.$$
Let us define
$$\hat{\Omega}:=\bigcup_{j\in\N}\hat{\Omega}^{(j)}.$$
We remark that the above union is increasing and that
$$P(\hat{\Omega};B_j(0))=P(\hat{\Omega}^{(j)};B_j(0))\quad\forall j\in\N;$$
then
$$P(\hat{\Omega})=\lim_{j\to+\infty}P(\hat{\Omega};B_j(0))=\lim_{j\to+\infty}P(\hat{\Omega}^{(j)};B_j(0))\le p$$
and $|\hat{\Omega}|<\infty$, i.e. $\hat{\Omega}\in\tilde{\mathcal{A}}_p$ is an admissible set for Problem \eqref{eq:relaxminbuck}.

It remains to prove that $u=0$ a.e. in $\hat{\Omega}^c$.

Let us fix $\varepsilon>0$. Since $u_{n_k}(\cdot+y_k)\to u$ strongly in $L^2(\R^d)$ and, for any $j\in\N$, $\left|\left(\Omega_{n_k}-y_k\cap B_j(0)\right)\setminus\hat{\Omega}^{(j)}\right|\to0$, then, for $k\in\N$ sufficiently large, it holds
$$\int_{(\Omega_{n_k}-y_k\cap B_j(0))\setminus\hat{\Omega}^{(j)}}u^2_{n_k}(\cdot+y_k)\:dx<\varepsilon.$$
Therefore
\begin{align*}
\int_{B_j(0)\setminus\hat{\Omega}}u^2\:dx&=\int_{B_j(0)\setminus\hat{\Omega}^{(j)}}u^2\:dx\le\liminf_{k\to+\infty}\int_{B_j(0)\setminus\hat{\Omega}^{(j)}}u^2_{n_k}(\cdot+y_k)\:dx\\
&=\liminf_{k\to+\infty}\int_{(\Omega_{n_k}-y_k\cap B_j(0))\setminus\hat{\Omega}^{(j)}}u^2_{n_k}(\cdot+y_k)\:dx\le\varepsilon.
\end{align*}
We thus conclude that
$$\int_{B_j(0)\setminus\hat{\Omega}}u^2\:dx=0\quad\forall j\in\N$$
and then, by taking the supremum over $j\in\N$,
$$\int_{\R^d\setminus\hat{\Omega}}u^2\:dx=0,$$
which proves that $u=0$ a.e. in $\hat{\Omega}^c$.

Then, since $u_{n_k}(\cdot+y_k)\to u$ strongly in $H^1(\R^d)$ as proved above, recalling \eqref{eq:minimizingsequence} we finally have
\begin{align*}
\tilde{\Lambda}_1(\hat{\Omega})\le\frac{\displaystyle\int_{\R^d}(\Delta u)^2\:dx}{\displaystyle\int_{\R^d}|\nabla u|^2\:dx}\le\frac{\displaystyle\liminf_{k\to+\infty}\int_{\R^d}(\Delta u_{n_k})^2\:dx}{\displaystyle\lim_{k\to+\infty}\int_{\R^d}|\nabla u_{n_k}|^2\:dx}=\liminf_{k\to+\infty}\frac{\displaystyle\int_{\R^d}(\Delta u_{n_k})^2\:dx}{\displaystyle\int_{\R^d}|\nabla u_{n_k}|^2\:dx}=\inf_{\Omega\in\tilde{\mathcal{A}}_p}\tilde{\Lambda}_1(\Omega),
\end{align*}
concluding the theorem.
\end{proof}
\end{theorem}


The choice of the weaker framework of measurable sets leads us to choose a different approach to the connectedness of the admissible sets: indeed, given $\Omega\in\tilde{\mathcal{A}}_p$, for every `cracked version' $\Omega'$ of $\Omega$ one has $\tilde{\Lambda}_1(\Omega')=\tilde{\Lambda}_1(\Omega)$; in particular, the same equality holds if we consider cracks splitting the set in two connected components lying at zero distance, e.g. if $\Omega\subset\R^3$ is a ball and $\Omega'$ is obtained by removing from $\Omega$ a maximal disk. In other words, it does not make sense to talk about connected components in the canonical sense, even for open sets. We point out that the only connected components that can be treated in a classical way are those lying at positive distance (since $\tilde{\Lambda}_1(\Omega)$ is not invariant under relative translations of connected components, unless they remain at positive distance). In view of that, we need to introduce the following definition. 

\begin{definition}[well separated sets]\label{def:welsep}
Let $A,B\subseteq\mathbb{R}^d$. We say that $A$ and $B$ are well separated if there exist two open sets $E_A,E_B$ and two negligible sets $N_A\subset A$, $N_B\subset B$ such that
$$(A\setminus N_A)\subseteq E_A,\quad (B\setminus N_B)\subseteq E_B,\quad \text{dist}(E_A,E_B)>0.$$
\end{definition}

As we expected from the concentration-compactness argument in Theorem \ref{teo:relax}, it can not happen that the optimal set is split in two or more well separated set of positive measure (dichotomy does not occur). We now show that every optimal set for Problem \eqref{eq:relaxminbuck} is `connected' in a generalized sense. The proof follows a standard argument for counting the connected components in shape optimization, with the only difference that in our framework $\tilde{\Lambda}_1$ is an infimum and not a minimum, in general.

\begin{proposition}[generalized connectedness of the minimizers]\label{pro:connect}
Every solution $\Omega$ of Problem \eqref{eq:relaxminbuck} is `connected' in the sense of Definition \ref{def:welsep}, i.e. if $\Omega$ is union of well separated sets, only one has positive Lebesgue measure.
\begin{proof}
Let us suppose that $\Omega=\Omega_1\cup\Omega_2$, where $\Omega_1$ and $\Omega_2$ are well separated sets of positive measure. Let $\varepsilon>0$ and let $u^\varepsilon\in\tilde{H}^2_0(\Omega)$ an $\varepsilon$-eigenfunction for $\tilde{\Lambda}_1(\Omega)$. Let us define
$$u^\varepsilon_1:=\begin{cases}
u^\varepsilon &\text{in $\Omega_1$}\\
0 &\text{in $\Omega_1^c$}
\end{cases}\quad,\quad
u^\varepsilon_2:=\begin{cases}
u^\varepsilon &\text{in $\Omega_2$}\\
0 &\text{in $\Omega_2^c$}
\end{cases}.
$$
Since $\Omega_1$ and $\Omega_2$ lie at positive distance, then $u^\varepsilon_1\in \tilde{H}^2_0(\Omega_1)$ and $u^\varepsilon_2\in \tilde{H}^2_0(\Omega_2)$ and so they can be used as test functions for $\tilde{\Lambda}_1(\Omega_1)$ and $\tilde{\Lambda}_1(\Omega_2)$ respectively. In view of the choice of $u^\varepsilon$, by \eqref{eq:uepsilon} we have
\begin{align*}
\tilde{\Lambda}_1(\Omega)+\varepsilon&>\frac{\displaystyle\int_{\Omega_1}(\Delta u_1^\varepsilon)^2\:dx+\int_{\Omega_2}(\Delta u_2^\varepsilon)^2\:dx}{\displaystyle\int_{\Omega_1}|\nabla u_1^\varepsilon|^2\:dx+\int_{\Omega_2}|\nabla u_2^\varepsilon|^2\:dx}\\
&\ge\min\left\{\frac{\displaystyle\int_{\Omega_1}(\Delta u_1^\varepsilon)^2\:dx}{\displaystyle\int_{\Omega_1}|\nabla u_1^\varepsilon|^2\:dx},\frac{\displaystyle\int_{\Omega_2}(\Delta u_2^\varepsilon)^2\:dx}{\displaystyle\int_{\Omega_2}|\nabla u_2^\varepsilon|^2\:dx}\right\}\ge\min\left\{\tilde{\Lambda}_1(\Omega_1),\tilde{\Lambda}_1(\Omega_2)\right\}
\end{align*}
where we used inequality \eqref{eq:rapporto}. In view of the arbitrariness of $\varepsilon>0$, either $\tilde{\Lambda}_1(\Omega_1)\le\tilde{\Lambda}_1(\Omega)$ or $\tilde{\Lambda}_1(\Omega_2)\le\tilde{\Lambda}_1(\Omega)$. Let us suppose the first case; then, dilating $\Omega_1$ by a factor $t>1$ in such a way that $P(t\Omega_1)=p$, we get $\tilde{\Lambda}_1(t\Omega_2)<\tilde{\Lambda}_1(\Omega)$ contradicting the minimality of $\Omega$.
\end{proof}
\end{proposition}

The previous result about the generalized connectedness of the optimal measurable sets can be proven identically even if we replace the perimeter constraint with the measure constraint.

\medskip

Once we have assured the existence of optimal shapes in this weak setting, we would like to show that weak solutions are in fact open solutions. To this aim, we follow an approach proposed in \cite{deve}, introducing the following definition.

\begin{definition}[perimeter supersolution]
Let $\Omega\subset\R^d$ a set of finite perimeter. We say that $\Omega$ is a \emph{perimeter supersolution} if $|\Omega|<+\infty$ and, for every $\tilde{\Omega}\supset \Omega$ of finite perimeter, we have $P(\tilde{\Omega})\ge P(\Omega)$.
\end{definition}

The following result is immediate.

\begin{proposition}
If $\Omega\subset\R^d$ is a solution for Problem \eqref{eq:relaxminbuck}, then $\Omega$ is a perimeter supersolution.
\begin{proof}
Let $\tilde{\Omega}\subset\R^d$ be a set of finite perimeter such that $\tilde{\Omega}\supset\Omega$. In view of the decreasing  monotonicity of $\tilde{\Lambda}_1(\cdot)$, it holds $\tilde{\Lambda}_1(\tilde{\Omega})\le\tilde{\Lambda}_1(\Omega)$. On the other hand, by using the equivalent penalized version of Problem \eqref{eq:relaxminbuck}, i.e. Problem \eqref{eq:relaxpenalized}, in view of the optimality of $\Omega$ for some $\beta>0$ we obtain
$$\tilde{\Lambda}_1(\Omega)+\beta P(\Omega)\le \tilde{\Lambda}_1(\tilde{\Omega})+\beta P(\tilde{\Omega})\le\tilde{\Lambda}_1(\Omega)+\beta P(\tilde{\Omega}),$$
i.e. $P(\Omega)\le P(\tilde{\Omega})$.
\end{proof}
\end{proposition}

Perimeter supersolutions enjoy good properties for our purposes; one of those is the following density estimate.

\begin{definition}[exterior density estimate]
Let $\Omega\subset\R^d$ a set of finite perimeter. We say that $\Omega$ \emph{satisfies an exterior density estimate} if there exists a positive dimensional constant $c=c(d)$ such that, for every $x\in\R^d$, one of the following situations occurs:
\begin{itemize}
\item[(i)] there exists $r>0$ such that $B_r(x)\subset\Omega$ a.e.;
\item[(ii)] for every $r>0$, it holds $|B_r(x)\cap\Omega^c|\ge c|B_r(x)|.$
\end{itemize}
\end{definition}

The next results link the previous density estimate with the perimeter supersolutions, ensuring that there exist open optimal shapes for Problem \eqref{eq:relaxminbuck} and that such an open solution is, in fact, a solution for Problem \eqref{eq:minbuck}. The proof of the following proposition is omitted, as it can be found in \cite[Lemma 4.4]{deve}.

\begin{proposition}
Let $\Omega\subset\R^d$ be a perimeter supersolution. Then, $\Omega$ satisfies an exterior density estimate. In particular, if $\Omega\subset\R^d$ is a solution of \eqref{eq:relaxminbuck}, then $\Omega$ satisfies an exterior density estimate.
\end{proposition}

The following crucial result states that the measure theoretic interior $\Omega_1$ of a perimeter supersolution $\Omega$ is an open set and that the test space $\tilde{H}^2_0(\Omega)$ is, in fact, the classical Sobolev space $H^2_0(\Omega_1)$. The proof is based on \cite[Proposition 4.7]{deve}, where the authors show the equality between the spaces $\tilde{H}^1_0(\Omega)$ and $H^1_0(\Omega_1)$.

\begin{proposition}\label{pro:omega1}
Let $\Omega\subset\R^d$ a set of finite perimeter satisfying an exterior density estimate. Then, the set of the points of density 1 for $\Omega$
$$\Omega_1=\left\{x\in\R^d:\exists\:\lim_{r\to 0^+}\frac{|\Omega\cap B_r(x)|}{|B_r(x)|}=1\right\}$$
is open. In particular, for every perimeter supersolution $\Omega$, $\Omega_1$ is open.

Moreover, it holds $\tilde{H}^2_0(\Omega)=H^2_0(\Omega_1)$ and, in particular, $\tilde{\Lambda}_1(\Omega)=\tilde{\Lambda}_1(\Omega_1)=\Lambda_1(\Omega_1)$.
\end{proposition}
\begin{proof}
The fact that $\Omega_1$ is open follows from the exterior density estimate.

To show the equality $\tilde{H}^2_0(\Omega)=H^2_0(\Omega_1)$ it is sufficient to prove that $\tilde{H}^2_0(\Omega)\subseteq H^2_0(\Omega_1)$. Moreover, since $|\Omega\Delta\Omega_1|=0$, the equality $\tilde{H}^2_0(\Omega)=\tilde{H}^2_0(\Omega_1)$ holds and so we prove the inclusion $\tilde{H}^2_0(\Omega_1)\subseteq H^2_0(\Omega_1)$. 

Let $u\in\tilde{H}^2_0(\Omega_1)$, so, in particular $u\in H^1(\R^d)$, $u=0$ a.e. in $\Omega^c_1$ and thus $u\in\tilde{H}^1_0(\Omega_1)$. By \cite[Proposition 4.7]{deve}, since $\Omega$ is a perimeter supersolution, we have that $H^1_0(\Omega_1)=\tilde{H}^1_0(\Omega_1)$ and so $u\in H^1_0(\Omega_1)$. Then, by Proposition \ref{teo:adams9.1.3}, we have $u=0$ everywhere in $\Omega_1^c$ and so, in view of Proposition \ref{teo:evans4.4}, we get $\nabla u=0$ a.e. in $\Omega_1^c$. Moreover, $D_j u\in H^1(\R^d)$ for any $j=1,\ldots, d$, but this implies that $D_j u \in\tilde{H}^1_0(\Omega_1)=H^1_0(\Omega_1)$ and so that $\nabla u=0$ everywhere in $\Omega_1^c$. By Proposition \ref{teo:adams9.1.3} we get that $u\in H^2_0(\Omega_1)$.
\end{proof}

Now we are ready to show that weak minimizers are equivalent to minimizing open sets for Problem \eqref{eq:minbuck}.

\begin{theorem}[existence of an open solution and equivalence with the original problem]\label{teo:equiv}
Problem \eqref{eq:relaxminbuck} admits an open solution. In particular, every solution of Problem \eqref{eq:relaxminbuck} is equivalent to a solution of Problem \eqref{eq:minbuck}, in the sense that if $\Omega\subset\R^d$ is an open set solving Problem \eqref{eq:minbuck}, then it solves also Problem \eqref{eq:relaxminbuck} and, on the other hand, if $\Omega\subset\R^d$ is a set of finite perimeter solving Problem \eqref{eq:relaxminbuck}, then $\Omega_1$ is an open set solving Problem \eqref{eq:minbuck}.
\begin{proof}
The existence of an open solution for Problem \eqref{eq:minbuck} is assured by the fact that, for any solution $\Omega\subset\R^d$ of Problem \eqref{eq:relaxminbuck}, the set $\Omega_1\subset\R^d$ is open (Proposition \ref{pro:omega1}) and admissible for \eqref{eq:relaxminbuck}. Indeed $P(\Omega_1)=P(\Omega)$ since $|\Omega\Delta\Omega_1|=0$; moreover, $\tilde{H}^2_0(\Omega_1)=\tilde{H}^2_0(\Omega)$, so $\tilde{\Lambda}_1(\Omega_1)=\tilde{\Lambda}_1(\Omega)=\inf_{\Omega\in\tilde{\mathcal{A}}_p}\tilde{\Lambda}_1(\Omega)$.

Let us show now the equivalence between Problem \eqref{eq:minbuck} and Problem \eqref{eq:relaxminbuck}. Let $\Omega\subset\R^d$ be a minimizer for Problem \eqref{eq:minbuck}. If it was not a minimizer for Problem \eqref{eq:relaxminbuck}, there would exist a solution $A\in\tilde{\mathcal{A}}_p$ for Problem \eqref{eq:relaxminbuck} such that
$$\tilde{\Lambda}_1(A)<\tilde{\Lambda}_1(\Omega).$$
On the other hand, since $A$ is also a perimeter supersolution, $A_1$ is an open set admissible for Problem \eqref{eq:minbuck} and so
$$\tilde{\Lambda}_1(\Omega)\le\Lambda_1(\Omega)\le\Lambda_1(A_1)=\tilde{\Lambda}_1(A),$$
a contradiction.

On the other hand, if $\Omega\in\tilde{\mathcal{A}}_p$ is a solution for Problem \eqref{eq:relaxminbuck}, then for any open set $A\in\mathcal{A}_p$ one has
$$\Lambda_1(\Omega_1)=\tilde{\Lambda}_1(\Omega)\le\tilde{\Lambda}_1(A)\le\Lambda_1(A),$$
getting the minimality of $\Omega_1$ for Problem \eqref{eq:minbuck}. 
\end{proof}
\end{theorem}

\begin{proof}[Proof of Theorem \ref{teo:maintheorem}]
It is a straightforward consequence of Proposition \ref{pro:exis2}, Theorem \ref{teo:relax}, Proposition \ref{pro:connect} and Theorem \ref{teo:equiv}.
\end{proof}

\section{Existence of optimal shapes for the higher buckling eigenvalues: the case of convex sets}\label{sec:kconv}

The existence of optimal shapes for higher eigenvalues needs a more careful investigation in the framework of sets of finite perimeter. It seems necessary an analysis of the boundedness of minimizers in order to apply an inductive concentration compactness argument as applied for instance in \cite{deve}. At the moment we are not able to get the required boundedness of the optimal sets since for these fourth order problems the common tools of surgery introduced in the $H^1$-setting fail (to get an overview for the Dirichlet-Laplace eigenvalues see, for instance, \cite{deve} for the problem with perimeter constraint or \cite{mazpra} for the problem with volume constraint).

Nevertheless, Theorem \ref{pro:kconv} ensures the existence of minimizers for higher eigenvalues in the framework of convex sets. The variational argument used to prove this result is based on a standard application of the direct methods which is inspired by previous works in which higher eigenvalues for the Laplace operator are minimized among convex sets (see, for instance, \cite{cito} for the case of Robin eigenvalues or \cite{bub} for the Dirichlet case).

In order to apply the direct methods of the Calculus of variations we show that $\Lambda_h$ is lower semicontinuous with respect to the Hausdorff convergence.

\begin{proposition}[Lower semicontinuity of $\Lambda_{h}$ among convex sets]\label{pro:lsc}
Let $(\Omega_n)_n$ be a sequence of open convex sets converging to an open, non empty, convex set $\Omega$ in the Hausdorff topology and let $\Omega_n,\Omega$ be contained in a compact set $D\subset\R^d$. Then, for every $k\in\N$,
$$\Lambda_{h}(\Omega)\le\liminf_{n\to+\infty}\Lambda_{h}(\Omega_n).$$
\end{proposition}
\begin{proof}
Without loss of generality, we can assume $\sup_{n\in\N}\Lambda_{h}(\Omega_n)<+\infty$.

Let $V^n\subset H^2_0(\Omega_n)$ be an admissible $h$-dimensional space for the computation of $\Lambda_{h}(\Omega_n)$ such that
$$\Lambda_{h}(\Omega_n)=\max_{V^n}R_{\Omega_n}.$$
Let $\left\{u_1^n,\ldots,u_h^n\right\}\subset H^2_0(\Omega_n)$ a $H^1_0(\Omega_n)$-orthonormal basis for $V^n$; for every $i=1,\ldots,h$ it holds
$$\int_{\Omega_n}(\Delta u_i^n)^2\:dx=R_{\Omega_n}(u_i^n)\le\max_{V^n}R_{\Omega_n}=\Lambda_{h}(\Omega_n)<C.$$
Then, $\sup_{n}\|u_i^n\|_{H^2_0(\Omega_n)}=\sup_{n}\|u_i^n\|_{H^2_0(D)}<+\infty$ for every $i=1,\ldots,h$. So, for every $i=1,\ldots,h$, there exists $u_i\in H^2_0(D)$ such that $u_i^n\rightharpoonup u_i$ in $H^2_0(D)$. Moreover, $u_i^n\to u_i$ in $H^1(D)$ and $\Omega_n\to\Omega$ also in measure, so $u_1,\ldots,u_h\in H^1(\Omega)$.

Notice that $u_1,\ldots,u_h$ are linearly independent in $H^2_0(\Omega)$, since $\Omega_n$ converges to $\Omega$ also in measure; hence, the linear space $V:=\text{span}\left\{u_1,\ldots,u_h\right\}$ is a competitor for the computation of $\Lambda_{h}(\Omega)$. Let $w=\sum\alpha_iu_i$ realizing the maximum of the Rayleigh quotient $R_\Omega(\cdot)$ on $V$ and let $w_n:=\sum\alpha_iu_i^n\in V^n$. 
Let us observe that the Dirichlet integral at the denominator converges and the numerator is lower semicontinuous and so the Rayleigh quotient is lower semicontinuous as well. Since $w_n\in V^n$, we conclude that
\begin{align*}
\Lambda_{h}(\Omega)&\le\max_{V} R_\Omega=R_\Omega(w)\le\liminf_{n\to+\infty}R_{\Omega_n}(w_n)\le\liminf_{n\to+\infty}\max_{V^n}R_{\Omega_n}\\
&=\liminf_{n\to+\infty}\Lambda_{h}(\Omega_n),
\end{align*}
obtaining the required lower semicontinuity of the buckling eigenvalues.
\end{proof}

Now we are able to prove Theorem \ref{pro:kconv}. In order to apply the direct methods of the Calculus of Variations, we need a compactness property for a minimizing sequence $(\Omega_n)_n$. To this aim, we just need a careful analysis about the non degeneracy and the uniform boundedness of the sequence $(\Omega_n)_n$.

\begin{proof}[Proof of Theorem \ref{pro:kconv}]
Let $(\Omega_n)_n$ be a minimizing sequence of admissible open convex sets for Problem \eqref{eq:kconv} such that $\mathcal{H}^{d-1}(\partial\Omega_n)=p$. Let us show that, up to subsequences, $\Omega_n$ converges in the sense of Hausdorff (and then in measure) to a nonempty admissible open convex set $\Omega$ with $\mathcal{H}^{d-1}(\partial\Omega)=p$. Without loss of generality, up to translations and rotations, we can assume that 
$$\text{diam}(\Omega_n)=\mathcal{H}^1(\Omega_n\cap\{x_2=\ldots=x_d=0\})$$
and that
$$\min_{i=2,\ldots,d}\left(\max_{\Omega_n}x_i-\min_{\Omega_n}x_i\right)=\max_{\Omega_n}x_d-\min_{\Omega_n}x_d$$
i.e. the length of the one dimensional projection of $\Omega_n$ on the axis $x_1$ is equal to the diameter of $\Omega_n$ and the  projection on the axis $x_d$ has minimal length. We claim that $\sup_n\text{diam}(\Omega_n)<+\infty$ and that, up to subsequences,
\begin{equation}\label{eq:width}
\lim_n\left(\max_{\Omega_n}x_d-\min_{\Omega_n}x_d\right)>0.
\end{equation}
We prove \eqref{eq:width} arguing by contradiction. Let us suppose that the limit in \eqref{eq:width} is zero; in view of John's Ellipsoid Theorem \ref{teo:john}, there exists an ellipsoid $E_n$ such that, up to rotations and translations
$$E_n\subseteq\Omega_n\subseteq d E_n.$$
Then, since also the width of $E_n$ vanishes, we have
$$\Lambda_h(\Omega_n)\ge\Lambda_h(d E_n)=\frac{1}{d^2}\Lambda_h(E_n)\ge\frac{1}{d^2}\Lambda_1(E_n)\ge\frac{1}{d^2}\lambda_2(E_n)\to+\infty,$$
against the minimality of $\Omega_n$. Then \eqref{eq:width} holds.

To prove that the diameters of the $\Omega_n$ sets are uniformly bounded, we argue again by contradiction. Let us suppose that the sequence of the diameters is unbounded. Since the sets $\Omega_n$ are convex and uniformly bounded in measure, the product
$$\prod_{j=1}^d\left(\max_{\Omega_n}x_j-\min_{\Omega_n}x_j\right)$$
has to be uniformly bounded. In view of our assumptions, as the diameter of $\Omega_n$ tends to infinity, necessarily the first term of the product diverges. We deduce that at least the smallest term among the remaining $d-1$ terms has to vanish. In other words, we have
$$\lim_n\left(\max_{\Omega_n}x_d-\min_{\Omega_n}x_d\right)=0,$$
in contradiction with \eqref{eq:width}.

Then $(\Omega_n)_n$ is an equibounded sequence of convex sets. In view of Proposition \ref{prop:convexcontinuity}(v), $(\Omega_n)_n$ converges (up to subsequences) in the sense of Hausdorff to a bounded convex set $\Omega$; moreover, by Proposition \ref{prop:convexcontinuity}(ii), the convergence is also in measure. In addition, thanks to \eqref{eq:width}, the limit set $\Omega$ is not degenerate (i.e. it has positive measure) and
$$\mathcal{H}^{d-1}(\partial\Omega)=\lim_{n\to+\infty}\mathcal{H}^{d-1}(\partial\Omega_n)=p.$$
In view of Proposition \ref{pro:lsc} $\Omega$ is the required minimizer.
\end{proof}

We point out that in dimension $d=2$ Theorem \ref{pro:kconv} proves the existence of open minimizers for $\Lambda_h$ in the whole of $\R^d$. Indeed, the problem
$$\min\left\{\Lambda_h(\Omega):\Omega\subset\R^2\ \text{open}, |\Omega|<\infty,\ P(\Omega)\le p\right\}$$
reduces to the minimization problem among convex sets with prescribed boundary length $p$.  

Moreover, if the perimeter constraint in Problem \eqref{eq:kconv} is replaced by a volume constraint, the existence of optimal shapes for problem
$$\min\left\{\Lambda_h(\Omega):\Omega\subset\R^d\ \text{open and convex}, |\Omega|\le m\right\},$$
can be proved by using the same arguments in this section.

\section{Further remarks and open problems}\label{sec:open}

Once proved the existence of minimizers, some interesting questions arise about the regularity or the precise shape of the minimizers. For the first eigenvalue the two questions are related, as highlighted in \cite{AshBuc2003} for the planar case with measure constraint: provided that an optimal shape is regular enough, it must coincide with the disk of given measure. In our framework we start from a better situation, since optimal planar sets are convex. In this case (and, more generally, in the case of Problem \eqref{eq:kconv}) it seems necessary at least to remove the possible corners to get more regularity of the boundary. Unfortunately, at the moment the cutting techniques that are known in the $H^1$-setting (see, for instance, \cite{buconvex} for the Dirichlet-Laplace eigenvalues, or \cite{cito} for the Robin-Laplace eigenvalues) do not seem to apply since they are based on surgery arguments that are not available in the $H^2$-setting.

Another aspect which is worth to analyze is the regularity of the optimal sets in higher dimension. Due to the choice of the perimeter constraint, it would be interesting to understand if it was possible to see minimizers for Problem \eqref{eq:minbuck} as quasi-minimizers of the perimeter in the sense of De Giorgi, as done in \cite{deve}, in order to obtain that optimal open sets have $C^{1,\alpha}$ boundary up to a singular set whose dimension is less then or equal to $d-7$. To this aim, it seems necessary to prove that optimal sets are bounded. Unfortunately, as highlighted at the beginning of Section \ref{sec:kconv}, up to our knowledge, there are no available techniques to reach this goal at the moment. 

We conclude giving the following list of some open problems.

\begin{open.problem}
Are optimal shapes for Problem \eqref{eq:minbuck} smooth? Is it possible to remove the corners from the boundary of the convex minimizers in the planar case?
\end{open.problem}
\begin{open.problem}
Provided that an optimal shape for Problem \eqref{eq:minbuck} is smooth enough, can we prove that it is a ball, at least in the planar case?
\end{open.problem}
\begin{open.problem}
Are optimal shapes for Problem \eqref{eq:minbuck} bounded also in higher dimension?
\end{open.problem}
\begin{open.problem}
Do minimizers for $\Lambda_h$ exist among open sets with prescribed perimeter in higher dimension? Are they bounded?
\end{open.problem}

\bigskip

\paragraph*{\bf Declarations} S.C. has been partially supported by the ACROSS project Cod. ARS01-00702. A.L. has been partially supported by the Italian M.U.R. PRIN: grant number 2017KC8WMB.
The authors have no competing interests to declare that are relevant to the content of this article.


\bigskip

\begin{bibdiv}
	\begin{biblist}
	
	
	\bib{adams1999function}{book}{
  title={Function spaces and potential theory},
  author={Adams, D.R.}
  author={Hedberg, L. I.},
  volume={314},
  year={1999},
  publisher={Springer Science \& Business Media}
}

		\bib{AFP}{book}{
			author={Ambrosio, L.},
			author={Fusco, N.}
			author={Pallara, D.},
			title={Functions of bounded variation and free discontinuity problems},
			series={Oxford Mathematical Monographs},
			publisher={The Clarendon Press, Oxford University Press, New York},
			date={2000},
			pages={xviii+434},
		}

\bib{AshBuc2003}{article}{
  title={On the isoperimetric inequality for the buckling of a clamped plate},
  author={Ashbaugh, S.},
	author={Bucur, D.},
  journal={Zeitschrift f{\"u}r angewandte Mathematik und Physik ZAMP},
  volume={54},
  number={5},
  pages={756--770},
  year={2003},
  publisher={Springer}
}

\bib{buconvex}{article}{
  title={Regularity of optimal convex shapes},
  author={Bucur, D.},
  journal={Journal of Convex Analysis},
  volume={10},
  number={2},
  pages={501--516},
  year={2003},
  publisher={HELDERMANN VERLAG LANGER GRABEN 17, 32657 LEMGO, GERMANY}
}

\bib{bub}{book}{
title={Variational Methods in Shape Optimization Problems},
author={Bucur, D.},
author={Buttazzo, G.},
year={2005},
publisher={Birkhauser}
}


\bib{bbh}{article}{,
  title={Minimization of $\lambda_2(\Omega)$ with a perimeter constraint},
  author={Bucur, D.}
  author={Buttazzo, G.}
  author={Henrot, A.},
  journal={Indiana University mathematics journal},
  pages={2709--2728},
  year={2009},
  publisher={JSTOR}
}

\bib{bucvar}{article}{
  title={Global minimizing domains for the first eigenvalue of an elliptic operator with non-constant coefficients},
  author={Bucur, D.}
  author={Varchon, N.},
  journal={Electronic Journal of Differential Equations},
  volume={36},
  year={2000},
  pages={1--10},
  publisher={Texas State University, Department of Mathematics}
}

\bib{cito}{article}{
  title={Existence and regularity of optimal convex shapes for functionals involving the robin eigenvalues},
  author={Cito, S.},
  journal={J. Convex Anal},
  volume={26},
  pages={925--943},
  year={2019}
}

\bib{deve}{article}{
  title={Existence and regularity of minimizers for some spectral functionals with perimeter constraint},
  author={De Philippis, G.}
	author={Velichkov, B.},
  journal={Applied Mathematics \& Optimization},
  volume={69},
  number={2},
  pages={199--231},
  year={2014},
  publisher={Springer}
}


		
\bib{evans2018measure}{book}{
  title={Measure theory and fine properties of functions - Revised Edition},
  author={Evans, L. C.}
  author={Gariepy, R. F.},
  year={2015},
  publisher={CRC Press, Boca Raton, FL}
  series = {Textbooks in Mathematics},
  pages = {xiv+299},
  ISBN = {978-1-4822-4238-6},
}		



		
\bib{GGS}{book}{
  title={Polyharmonic boundary value problems: positivity preserving and nonlinear higher order elliptic equations in bounded domains},
  author={Gazzola, F.},
	author={Grunau, H.-C.},
	author={Sweers, G.},
  year={2010},
  publisher={Springer Science \& Business Media}
}	

\bibitem{joh}
F. John:
Extremum problems with inequalities as subsidiary conditions, in \emph{Studies and Essays Presented to R. Courant on his 60th Birthday, January 8, 1948}, Interscience Publishers, Inc., New York, N. Y., pp. 187-204 (1948)


\bib{lions84}{inproceedings}{
  title={The concentration-compactness principle in the Calculus of Variations. The locally compact case, part 1.},
  author={Lions, P.-L.},
  booktitle={Annales de l'Institut Henri Poincaré (C) Non Linear Analysis},
  volume={1},
  number={2},
  pages={109--145},
  year={1984},
  organization={Elsevier}
}

\bib{mazpra}{article}{
  title={Existence of minimizers for spectral problems},
  author={Mazzoleni, D.}
  author={Pratelli, A.},
  journal={Journal de Math{\'e}matiques Pures et Appliqu{\'e}es},
  volume={100},
  number={3},
  pages={433--453},
  year={2013},
  publisher={Elsevier}
}
		
\bib{sto16}{article}{
  title={Optimal shape of a domain which minimizes the first buckling eigenvalue},
  author={Stollenwerk, K.},
  journal={Calculus of Variations and Partial Differential Equations},
  volume={55},
  number={1},
  pages={5},
  year={2016},
  publisher={Springer}
}		
		
	
	\bib{sto2021}{article}{
  title={On the optimal domain for minimizing the buckling load of a clamped plate},
  author={Stollenwerk, K.},
  journal={arXiv preprint arXiv:2110.02545},
  year={2021}
}
	
	
	\end{biblist}
\end{bibdiv}
\end{document}